\documentclass[preprint,10pt]{elsarticle}

\usepackage{graphicx}
\usepackage{amssymb}

\usepackage{lineno}

\journal{ }

\usepackage{booktabs}  % professionally typeset tables
\usepackage{amsmath,amsfonts,amsthm,amscd}
\usepackage{textcomp}  % better copyright sign, among other things
\usepackage{xcolor}
\usepackage{lipsum}    % filler text
\usepackage{subfig}    % composite figures
\usepackage{mathrsfs}
\usepackage{enumerate}
\usepackage{multirow,mathdots}
\usepackage{tikz}
\usepackage{gensymb}
\usepackage{chngcntr}
\usepackage{bbm}

\newcommand{\SL}{\operatorname{SL}}
\newcommand{\GL}{\operatorname{GL}}
\newcommand{\SO}{\operatorname{SO}}

\newcommand{\Int}{\operatorname{Int}}

\newcommand{\id}{\operatorname{id}}

\newcommand{\diag}{\operatorname{diag}}

\newcommand{\pghk}{P_k \backslash G_k \slash H_k}
\newcommand{\pgh}{P \backslash G \slash H}
\newcommand{\bghk}{B_k \backslash G_k \slash H_k}
\newcommand{\bgh}{B \backslash G \slash H}
\newcommand{\intone}{\operatorname{Int}(\begin{smallmatrix} 0&1 \\ 1 & 0 \end{smallmatrix})}

\theoremstyle{plain}
\newtheorem{thm}[subsection]{Theorem}
\newtheorem{lemma}[subsection]{Lemma}
\newtheorem{prop}[subsection]{Proposition}
\newtheorem{cor}[subsection]{Corollary}

\theoremstyle{definition}
\newtheorem{defn}[subsection]{Definition}

\theoremstyle{remark}

\newtheorem{rem}[subsection]{Remark}

\newtheorem{eg}[subsection]{Example}
\newtheorem{notation}[subsection]{Notation}

\newsavebox{\smlmat}% Box to store smallmatrix content
\savebox{\smlmat}{$\left(\begin{smallmatrix}0&1\\1&0\end{smallmatrix}\right)$}
\newsavebox{\smlmatneg}% Box to store smallmatrix content
\savebox{\smlmatneg}{$\left(\begin{smallmatrix}0&1\\-1&0\end{smallmatrix}\right)$}
\newcommand{\st}{\,|\,}
\newcommand{\ds}{\displaystyle}

\newcommand{\R}{\mathbb{R}}
\newcommand{\C}{\mathbb{C}}
\newcommand{\Q}{\mathbb{Q}}

\DeclareMathOperator{\Aut}{Aut}
\DeclareMathOperator{\rank}{rank}
\DeclareMathOperator{\krank}{-rank}

\DeclareMathOperator{\cl}{cl}

\counterwithin{figure}{section}
\counterwithin{table}{section}

\begin{document}

\begin{frontmatter}

%% Title, authors and addresses

\title{Minimal Parabolic $k$-subgroups acting on Symmetric $k$-varieties Corresponding to $k$-split Groups}

\author{Mark Hunnell}

\address{Winston-Salem State University}

\begin{abstract}
Symmetric $k$-varieties are a natural generalization of symmetric spaces to general fields $k$.  We study the action of minimal parabolic $k$-subgroups on symmetric $k$-varieties and define a map that embeds these orbits within the orbits corresponding to algebraically closed fields.  We develop a condition for the surjectivity of this map in the case of $k$-split groups that depends only on the dimension of a maximal $k$-split torus contained within the fixed point group of the involution defining the symmetric $k$-variety.
\end{abstract}

%\begin{keyword}
%Science \sep Publication \sep Complicated
%%% keywords here, in the form: keyword \sep keyword
%
%%% MSC codes here, in the form: \MSC code \sep code
%%% or \MSC[2008] code \sep code (2000 is the default)
%
%\end{keyword}

\end{frontmatter}

%% main text
\section{Introduction}

Real reductive symmetric spaces are the homogeneous spaces $X:=G \slash H$ where $G$ is a reductive Lie group and $H$ is an open subgroup of the fixed point group of an involution of $G$.  The study of the representations of these spaces culminated in the work of Delorme \cite{delorme}, opening the door for the study of generalizations of the real reductive spaces.  

Let $k$ be a field of characteristic not 2.  The symmetric $k$-variety is the homogenous space $X_k := G_k \slash H_k$, where $G$ is a reductive algebraic group defined over $k$, $H$ is a $k$-open subgroup of the set of fixed points of a $k$-involution $\theta$, and $G_k$ (respectively $H_k$) denote the $k$-rational points of $G$ (resp. $H$).  The symmetric $k$-varieties have applications in diverse fields of mathematics including representation theory \cite{vogan}, geometry \cite{geometry}, the study of character sheaves \cite{lusztig}, and in the cohomology of arithmetic subgroups \cite{tongwang}.  Beginning in the 1980's, Helminck and Wang \cite{helmwang} commenced a study of the rationality properties of the symmetric $k$-varieties for arbitrary fields.  These have continued to be studied because of their applications, particularly in representation theory.  For this reason, it is natural to consider how the parabolic orbits on a symmetric variety (i.e. where $k$ is algebraically closed) are related to the orbits of a corresponding minimal parabolic $k$-subgroup acting on a symmetric $k$-variety.  Our approach is to embed the orbits over the symmetric $k$-variety within the orbits corresponding the algebraic closure of $k$ via a map which generalizes the complexification of orbits in the case $k = \R$:
$$\varphi: \pghk \hookrightarrow \pgh$$

After providing the relevant background and notation, we recall the necessary results to discuss the orbits of Borel and parabolic subgroups acting on symmetric varieties.  We then explain the relevant differences between the algebraically closed case and the general case.  Finally, we restrict our attention to the case $G$ is a $k$-split group and develop a condition for the surjectivity of the generalized complexification map.

\section{Background}

\begin{notation}
We follow the notation established in \cite{borel}, \cite{springer}, and \cite{humphreys}.  Throughout the paper $G$ will denote a connected reductive algebraic group, $\theta$ a group involution of $G$ that leaves the $k$-rational points invariant and $H$ a $k$-open subgroup of $G^{\theta} = \{ g \in G \st \theta(g) =g \}$.  The variety $G \slash H$ is called the symmetric variety and $G_k \slash H_k$ is called the symmetric $k$-variety.
\end{notation}

Define a map $$\tau : G \rightarrow G, \ \tau(x) = x\theta(x)^{-1}$$
Denote the image of $\tau$ by $Q$, then $\tau$ induces an isomorphism between $G\slash H$ and $Q$ as well as an isomorphism between $G_k \slash H_k$ and $Q_k$.  It is sometimes more convenient in calculation to let $H$ act from the left, in this case $\tau(x) = x^{-1}\theta(x)$.

Let $\Aut(G)$ denote the set of group automorphisms $G \rightarrow G$, and for a subgroup $K \subset G$ we will use $\Aut(G,K)$ to denote the set of automorphisms of $G$ which leave $K$ invariant.  In particular, we are concerned with the order 2 elements of $\Aut(G,G_k)$.  For $g \in G$, we will use $\Int(g)$ to denote the inner automorphism corresponding to $g$. We denote the algebraic closure of $k$ by $\overline k$.

Recall that a torus $T$ of $G$ is a connected semisimple abelian subgroup.  Let $T$ be a torus, then $N_G(T) $ will denote the normalizer of $T$ in $G$, $Z_G(T)$ will denote the centralizer of $T$ in $G$, and $W_G(T) = N_G(T) \slash Z_G(T)$ will denote the Weyl group.  The classification of orbits of minimal parabolic subgroups acting on symmetric $k$-varieties relies on a quotient of the Weyl group by elements having representatives in the fixed point group $H$, and we denote this set by $W_H(T) = N_H(T)/Z_H(T)$.  These groups all have analogues for a group defined over $k$, replacing the group in the definition of in $N_G(T), \ Z_G(T), \ W_G(T), \text{ and }  \ W_H(T)$ with its associated $k$-rational points, we obtain definitions for $N_{G_k}(T)$, $ Z_{G_k}(T)$, $ W_{G_k}(T)$,  and $ W_{H_k}(T)$.  

Let $\varphi \in \Aut(G,T)$, then $T$ can be decomposed via its $\varphi$ (Lie algebra) eigenspaces, i.e. $T = T_{\varphi}^+T_{\varphi}^-$ where $T_{\varphi}^+ = \{ t \in T \st \theta(t) =t \}^{\degree}$ and $T_{\varphi}^- = \{ t \in T \st \theta(t) =t^{-1} \}^{\degree}$, where $K^{\degree}$ denotes the identity component of subgroup $K \subset G$.  The product map $$\mu: T_{\varphi}^+ \times T_{\varphi}^- \rightarrow T, \ \mu(t_1,t_2) = t_1t_2$$ is a separable isogeny.  In fact, $T_{\varphi}^+ \cap T_{\varphi}^-$ is an elementary abelian 2-group.  Of particular interest will be the case when $\varphi = \theta$, in this case we will use $T^+$ for $T_{\theta}^+$ and $T^-$ for $T_{\theta}^-$.  

Maximal tori play a fundamental role in the description of the structure of symmetric $k$-varieties.  A torus is maximal if it is properly contained in no other torus.  It is a fact that all such tori are conjugate under $G$.

We denote the root system of a torus $T$ by $\Phi(T)$, its positive roots by $\Phi^+(T)$, and a basis by $\Delta$.  A torus $T$ is called $\theta$-split if $\theta(t) = t^{-1}$ for all $t \in T$.  From \cite{richardson} we know that if $A$ is a maximal $\theta$-split torus, then $\Phi(A)$, the set of roots of $A$ in $G$, is a root system with Weyl group $W(A) = N_G(A) \slash Z_G(A)$.  Recall that a $k$-torus is $k$-split if it can be diagonalized over the base field $k$.  We will call a $k$-torus $(\theta,k)$-split if it is $\theta$-split and $k$-split.  These tori yield a natural root system for the symmetric $k$-variety $G_k \slash H_k$ since a maximal $(\theta,k)$-split torus $A$ of $G$ has a root system $\Phi(A)$ \cite{helmwang} with Weyl group $N_{G_k}(A) \slash Z_{G_k}(A)$.  Additionally this root system can be obtained by restricting the roots of a maximal torus of $G$ containing $A$.  

The rank of a group $G$ is the dimension of a maximal torus $T \subset G$ and the $k\krank$ of a group is the dimension of a maximal $k$-split torus.  A group is called $k$-split if the $k\krank$ is equal to the rank.

%\section{Parabolic and Borel Subgroups}
Given a Borel subgroup $B$, fix a maximal torus $T\subset B$.  Recall that there is a one-to-one correspondence between subsets of the simple roots and parabolic subgroups containing $B$.  A parabolic subgroup defined over $k$ is called a parabolic $k$-subgroup.  A parabolic $k$-subgroup is minimal if it properly contains no other parabolic $k$-subgroups, thus over algebraically closed fields the minimal parabolic subgroups are Borel subgroups.  For non-algebraically closed fields it may be the case that nontrivial minimal parabolic $k$-subgroups do not exist.  Parabolic subgroups are also self-normalizing, and can be decomposed via $P=LU$, where $U$ is the unipotent radical $R_u(P)$ and $L$ is the Levi factor of $P$.  In the case that $P=B$ is a Borel subgroup, this decomposition simplifies to $B=TU$, where $T \subset B$ is a maximal torus and $U \subset B$ is unipotent.

To obtain a characterization of the orbits $\bgh$, one can separately consider $B$-orbits on $G \slash H$, $H$-orbits on $G \slash B$, or $(B,H)$-orbits on $G$.  We outline these results in this section.  Let $T \subset G$ be a maximal $\theta$-stable torus, which exist by \cite{steinberg}.  We will first consider the $H$-orbits.  Let $\mathscr B$ denote the variety of all Borel subgroups of $G$, then we can identify $G \slash B$ with $\mathscr B$ since all Borel subgroups are conjugate over $G$, and let $\mathscr C$ denote the set of pairs $(B^{\prime},T^{\prime})$ where $T^{\prime}$ is a $\theta$-stable maximal torus contained in the Borel subgroup $B^{\prime}$.  $H$ acts on both $\mathscr B$ and $\mathscr C$ by conjugation, denote the sets of $H$-orbits by $\mathscr B \slash H$ and $\mathscr C \slash H$ respectively.  The $H$-orbits on $\mathscr C$ can be analyzed in two steps.  First, let $I$ be a set which indexes the $H$-conjugacy classes of $\theta$-stable maximal tori of $G$.  This allows us to choose representatives $\{T_i\}_{i \in I}$ for the $H$-conjugacy classes of $\theta$-stable maximal tori of $G$.  Next we consider the Borel subgroups containing each $T_j$ that are not $H$-conjugate, which are determined by $W_G(T_j) \slash W_H(T_j)$.  Thus the $H$-orbits on $\mathscr C$ are in correspondence with $\cup_{i \in I} W_G(T_i) \slash W_H(T_i)$.

G acts on $G \slash H \cong Q$ (from the left) via the \emph{$\theta$-twisted action}, i.e. $g * q  = gq\theta(g)^{-1}$.  Thus the $B$-orbits on $G \slash H$ can be viewed as $B$-cosets in $Q$, which we denote $B \backslash Q$.

The $(H,B)$-orbits on $G$ are the same as the $B \times H$-orbits on $G$ and the action is given by $(b,h) * g := bgh^{-1}$.   From \cite{helmwang} we know that every $U$-orbit on $G \slash H$, where U is the unipotent component of $B$, meets $N_G(T)$.  Let $\mathcal V = \{ g \in G \st \tau(g) \in N_G(T) \}$, then $\mathcal V$ is stable under left multiplication by $N_G(T)$ and right multiplication by $H$.  We denote by $V$ the $T\times H$-orbits on $\mathcal V$, which in fact parameterize the $(B \times H)$ orbits on $G$.  In fact, these characterizations are isomorphic.

%Borel showed that all of these characterizations are isomorphic:
\begin{thm}[\cite{springerpaper}, \textsection 4]
\label{springerthm}  Let $B$ be a Borel subgroup of $G$ and $\{ T_i \}_{i \in I}$ a set of representatives of the $H$-conjugacy classes of $\theta$-stable maximal tori in $G$.  Then
$$B \backslash G \slash H \cong \mathscr B \slash H \cong \ds \bigcup_{i \in I} W_G(T_i) \slash W_H(T_i) \cong \mathscr C \slash H \cong B \backslash Q \cong V$$
\end{thm}

One can endow the set of double cosets $ B \backslash G \slash H$ with a partial order that generalizes the usual Bruhat order on a connected reductive algebraic group, this was studied in \cite{richspring1} and \cite{richspring2}.  Let $\mathscr O_1 = Bg_1H$ and $\mathscr O_2 = Bg_2H$, then $\mathscr O_1 \leq \mathscr O_2$ if $\mathscr O_1 \subset \cl(\mathscr O_2)$, where $\cl(\mathscr O_2)$ denotes the Zariski closure of $\mathscr O_2$.

The Bruhat order is a refinement of the order on the $I$-poset, which we now introduce.  Let $I$ be an index set of $H$-conjugacy classes of the $\theta$-stable maximal tori in $G$, and denote these conjugacy classes by $\{ [T_i] \}_{i \in I}$.  We define $[T_i] \leq [T_j]$ if $T_i^- \subset T_j^-$ for some representatives $T_i$ of $[T_i]$ and $T_j$ of $[T_j]$.  Given a set of representatives $\{ T_i \}_{i \in I}$ for $\{ [T_i] \}_{i \in I}$, we write $T_i \leq T_j$ if $[T_i] \leq [T_j]$.  Furthermore, we can associate poset diagrams to the orbit decompositions.  Since the Bruhat poset is a refinement of the $I$-poset, we will call the Bruhat diagram obtained from an $I$-diagram the expansion of the $I$-diagram.

\section{Parabolic Subgroups acting on Symmetric $k$-varieties}
In this section we collect the necessary results and set the notation needed to discuss the generalized complexification map in section \ref{GC}.  Following \cite{gensymmspaces}, we outline the action of parabolic subgroups on the symmetric $k$-varieties in a number of cases.  While the approach is similar in each case, there are important differences which effect the analysis of the generalized complexification map.

\subsection{\boldmath$k = \bar{k}$, $P$ a parabolic subgroup acting $G \slash H$}

This section summarizes the results of Brion and Helminck \cite{brion} for the generalized Bruhat decomposition.  The important difference in this case is that for a fixed parabolic $P$ we are not assured that the set of $G$-conjugates of $P$ includes all parabolic subgroups of $G$.  In fact, if we identify parabolic subgroups $P_1$ and $P_2$ with their associated subsets of a base for the root system $\Gamma_1, \Gamma_2 \subset \Delta$, we see that $P_1$ and $P_2$ are $G$-conjugate if and only if $\Gamma_1 = \Gamma_2$.
Let $\mathscr P$ denote the variety of all parabolic subgroups of $G$ and let $\mathscr D$ denote the set of triples $(P,B,T)$, where $B$ is a Borel subgroup of $P$ such that $(P \cap H)B$ is open in $P$ and $T$ is a $\theta$-stable maximal torus of $B$.  Fix a parabolic subgroup $P$ and let $\mathscr P^P$ denote the set of $G$ conjugates of $P$ and let $\mathscr D^P$ denote the set of triples $(P^{\prime}, B^{\prime}, T^{\prime}) \in \mathscr D$ such that $P^{\prime} \in \mathscr P^P$.  $G$ acts on $\mathscr P$, $\mathscr P^P$, $\mathscr D$, and $\mathscr D^P$ via conjugation; denote the $H$-orbits on these sets by  $\mathscr P \slash H$, $\mathscr P^P \slash H$, $\mathscr D \slash H$, and $\mathscr D^P \slash H$ respectively.  
\begin{thm} There is a bijective map $\mathscr P \slash H \rightarrow \mathscr D \slash H$.  
\label{PstableT}
\end{thm}

Every maximal $k$-split torus of $G$ is conjugate under $G_k$, the $k$-rational points of $G$.  Furthermore, every minimal parabolic $k$-subgroup of $G$ contains a maximal $k$-split torus.  Let $A$ be a maximal $k$-split torus.  We wish to characterize all minimal parabolic $k$-subgroups which contain $A$.  As a generalization of the algebraically closed case, we have the following lemma:
\begin{prop}Suppose $A_1$ and $A_2$ are two maximal $k$-split tori, and $P_1$ and $P_2$ are two minimal parabolic $k$-subgroups containing $A_1$ and $A_2$ respectively.  Then the element of $G_k$ that conjugates $A_1$ to $A_2$ also conjugates $P_1$ to $P_2$.
\end{prop}

We now fix a parabolic subgroup $P$, and let $T$ be the $\theta$-stable maximal torus occurring in the image of $P$ under the bijection in Theorem \ref{PstableT}.  Let $x_1, x_2, ... , x_r \in G$ such that $T_1 = x_1Tx_1^{-1}, ..., T_r=x_rTx_r^{-1}$ are representatives for the $H$-conjugacy classes of the $\theta$-stable maximal tori occuring in $\mathscr D$.  Let $P_1 = x_1Px_1^{-1}, ..., P_r=x_rPx_r^{-1}$ and denote by $W_{P_i}(T_i)$ the Weyl group of $P_i$.  One then sees that for each $T_i$ the $H$-conjugacy classes of $(P^{\prime},B^{\prime},T_i) \in \mathscr D^P$ are in bijection with $W_{P_i}(T_i) \backslash W_G(T_i) \slash W_H(T_i)$.  As before let $\mathcal V = \{ g\in G \st \tau(g) \in N_G(T) \}$ and let $\mathcal V^P =\{ g \in \mathcal V \st BgH \text{ is open in } PgH \}$.  Then the actions of $B,H$ on $\mathcal V$ extend to actions of $P,H$ on $\mathcal V^P$, and every $P \times H$-orbit on $G$ meets $\mathcal V^P$ in a unique $(T,H)$ double coset.  Let $V^P$ denote the $(T,H)$-double cosets in $\mathcal{V}^P$.
We can now generalize Springer's theorem charactering $ B \backslash G \slash H$ double cosets to the case of a general parabolic subgroup.

\begin{thm} There is a bijective map from the set of $H$-orbits in $\mathscr P$ onto the set of $H$-conjugacy classes of triples $(P,B,T) \in \mathscr D$.  Moreover for a fixed parabolic subgroup $P$, we have \[
 P \backslash G \slash H \cong \mathscr P(G)^P \slash H \cong \mathscr D^P \slash H \cong \displaystyle \bigcup_{i=1}^r W_{P_i}(T_i) \backslash W_G(T_i) \slash W_H(T_i) \cong V^P \]
\end{thm}

\subsection{Orbits Over Non-algebraically Closed Fields}
We now turn to the double cosets which we study for the remainder.  If we restrict our attention to parabolic subgroups defined over $k$ that are minimal we can obtain a characterization very similar to the case of a Borel subgroup acting on $G \slash H$.  In this case we are assured of the existence of $\theta$-stable Levi factor.

\begin{thm}[\cite{helmwang}]
\label{Pthetastable}
Let $P$ be a minimal parabolic $k$-subgroup with unipotent radical $U$, then $P$ contains a $\theta$-stable maximal $k$-split torus, unique up to an element of $(H\cap U)_k$.
\end{thm}

Let $P$ be a minimal parabolic $k$-subgroup and let $G$ and $H$ be defined as before.  With a few adjustments, one can construct a characterization of the $\pghk$ orbits in several equivalent ways by considering $P_k$ orbits on $Q_k$, $H_k$ orbits on $P_k \backslash G_k$, or $P_k \times H_k$ orbits on $G_k$.  

We start with the $H_k$-orbits on $P_k \backslash G_k$.  Let $\mathscr P_k$ denote the variety of all minimal parabolic $k$-subgroups of $G$, then we have that $P_k \backslash G_k$ is isomorphic to $\mathscr P_k$.  $G_k$ acts on $\mathscr P_k$ by conjugation, so we can identify the double cosets with the $H_k$-orbits on $\mathscr P_k$, denote these orbits by $\mathscr P_k \slash H_k$.

Let $\mathscr C_k$ denote the set of all pairs $(P_k^{\prime}, A_k^{\prime})$, where $P_k^{\prime}$ is a minimal parabolic $k$-subgroup and $A_k^{\prime}$ is a $\theta$-stable maximal $k$-split torus contained in $P_k$.  $G_k$ acts on $\mathscr C_k$ by conjugation in both coordinates, i.e. $g * (P_k^{\prime}, A_k^{\prime}) = (gP_k^{\prime}g^{-1}, gA_k^{\prime}g^{-1})$.  We can analyze the $H_k$ orbits on $\mathscr C_k$ (denoted $\mathscr C_k \slash H_k$) in two steps; first we consider the $H_k$-conjugacy classes of $\theta$-stable maximal $k$-split tori and choose a set of representatives for these conjugacy classes, and second for each representative of an $H_k$-conjugacy class we consider the set of minimal parabolic $k$-subgroups that contain the representative but are not conjugate via $H_k$.  This allows one to identify $\mathscr C_k \slash H_k$ with $\ds \cup_{i \in I} W_{G_k}(A_i) \slash W_{H_k}(A_i)$, where $\{ A_i \}_{i \in I}$ is a set of representatives for the $H_k$ conjugacy classes of $\theta$-stable maximal $k$-split tori.

$P_k$ acts on $G_k \slash H_k$ via the $\theta$-twisted action.  Let $A \subset P$ be a maximal $k$-split torus.  Then as in the case of a Borel subgroup acting on the symmetric space over an algebraically closed field, we have that the orbit of the unipotent radical of $P_k$ meets $N_{G_k}(A)$.  Let $\mathcal V_k = \{ x \in G_k \st \tau(x) \in N_{G_k}(A) \}$, then we can identify the $P_k$-orbits on $G_k \slash H_k$ with the $Z_{G_k}(A) \times H_k$-orbits on $\mathcal V_k$.  We denote these orbits by $V_k$.  

\begin{thm}
\label{pghk}
For a minimal parabolic $k$-subgroup $P$and $\{A_i\}_{i \in I}$ a set of representatives for the $H_k$-conjugacy classes of $\theta$-stable maximal $k$-split tori, we have 
$$ \pghk \cong \mathscr P_k \slash H_k \cong  \bigcup_{i \in I} W_{G_k}(A_i) \slash W_{H_k}(A_i) \cong P_k \backslash Q_k \cong V_k$$
\end{thm}

We observe that this characterization is in direct analogy with the case of a Borel subgroup acting on the symmetric space and simpler than the case of a general parabolic subgroup acting on the symmetric space.  This is because a minimal parabolic $k$-subgroup contains a $\theta$-stable maximal $k$-split torus in light of Theorem \ref{Pthetastable}.

\begin{eg} \label{sl2 compute}
Let $G=\SL(2,\C)$ with real form $G_{\R} = \SL(2,\R)$.  Let $T$ denote the group of diagonal matrices and $P$ the set of upper triangular matrices.  
\begin{enumerate}[(a)]
\item Let $\sigma = \Int(\begin{smallmatrix} 0&1 \\ 1 & 0 \end{smallmatrix}) $.  Then $H_{\R} = \{ \left( \begin{smallmatrix} a &b \\ b & a \end{smallmatrix} \right) \st a^2 - b^2 =1 \ a,b \in \R \}^{\textdegree}$ is connected and abelian and is diagonalizable by an orthogonal matrix.  We compute the Weyl group elements 
$$ W_{G_{\R}}(T) = W_{G_{\R}}(H)= \left\{ \left( \begin{matrix} 1 & 0 \\ 0 & 1 \end{matrix} \right) , \ \left( \begin{matrix} 0 & 1 \\ -1 & 0 \end{matrix} \right) \right\} $$
The nonindentity element has a representative in $H$, given by $\left( \begin{smallmatrix} 0 & i \\ i & 0 \end{smallmatrix} \right)$, but this representative is not in $H_{\R}$.  Thus we conclude $|W_{G_{\R}}(T) \slash W_{H_{\R}}(T)| = |W_{G_{\R}}(H) \slash W_{H_{\R}}(H)| =2$, and there are 4 orbits in $P_{\R} \backslash G_{\R} \slash H_{\R}$.

\item Let $\theta = \Int(\begin{smallmatrix} 0&1 \\ -1 & 0 \end{smallmatrix}) $.  Then $H_{\mathbb{R}} = \{ \left( \begin{smallmatrix} a &b \\ -b & a \end{smallmatrix} \right) \st a^2 + b^2 =1, \ a,b \in \R \}$.  While $H_{\R}$ is connected and abelian, its eigenvalues are complex and thus it not an $\R$-split torus.  Therefore there is only $H_{\R}$-conjugacy class of $\theta$-stable maximal $k$-split tori and only one orbit in $P_{\R} \backslash G_{\R} \slash H_{\R}$.
\end{enumerate}
\end{eg}

As with the Bruhat decomposition, we can reformulate Theorem \ref{pghk} to obtain $$G_k = \bigcup_{i \in I}\bigcup_{w \in W_{G_k}(A_i) } P_k \dot{w} H_k, \ \text{where $\dot{w}$ is a representative of }w\in W_{G_k}(A_i)$$
Applying this observation to Example \ref{sl2 compute}(b) and the fact that $\id \in N_{G_{\R}}(T)$, we have that $G_{\R} = H_{\R}P_{\R}$, which is the well known Iwasawa decomposition of a real reductive group with Cartan Involution $\theta$.

While the characterization is similar, many of the properties from the case of an algebraically closed field do not hold.  For instance, the orbits in Theorem \ref{pghk}  are always finite over an algebraically closed field, but over general fields this condition is frequently not satisfied.

In several cases, however, the number of orbits is finite.  For algebraically closed fields this was proved by Springer \cite{springerpaper}, for $k=\R$ it was shown by Matsuki \cite{matsuki}, Rossman \cite{rossman}, and Wolf \cite{wolf}, and for general local fields the result is due to Helminck and Wang \cite{helmwang}.

Similar to the algebraically closed case, we can place an order on the $I$-poset. Let $I$ be an index set of $H_k$-conjugacy classes of the $\theta$-stable maximal $k$-split tori in $G$, and denote these conjugacy classes by $\{ [A_i] \}_{i \in I}$.  We define $[A_i] \leq [A_j]$ if $A_i^- \subset A_j^-$ for some representatives $A_i$ of $[A_i]$ and $A_j$ of $[A_j]$.  Given a set of representatives $\{ A_i \}_{i \in I}$ for $\{ [A_i] \}_{i \in I}$, we write $A_i \leq A_j$ if $[A_i] \leq [A_j]$.  As with maximal tori in the algebraically closed case, the order extends to all maximal $k$-split tori.  Thus each representative of an $H_k$-conjugacy class of $\theta$-stable maximal $k$-split tori with the same dimension of its split component correspond to nodes in the same level in the diagram associated to the order on the $I$-poset. 
\begin{eg} Consider $\SL(2,\Q_p)$, $p \equiv 1 \mod 4$.  Let $\theta = \intone$.  Then by \cite{beun} we know that there are four $H_k$-conjugacy classes of $\theta$-stable maximal $(\theta,\Q_p)$-split tori and one $H_k$-conjugacy classes of $\theta$-stable, $\theta$-fixed maximal $\Q_p$-split tori.  Thus the diagram for the $I$-poset is:
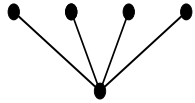
\begin{figure}[h]
\begin{center}
\resizebox{1in}{.5in}{
\begin{tikzpicture}
	\node[draw,circle, inner sep=4pt,fill] at (0,0) {};
	\node[draw,circle, inner sep=4pt,fill] at (2,0) {};
	\node[draw,circle, inner sep=4pt,fill] at (4,0) {};
	\node[draw,circle, inner sep=4pt,fill] at (6,0) {};
	%\node[draw,circle, inner sep=2pt,fill,red] at (0,-1) {};
	%\node[draw,circle, inner sep=2pt,fill,red] at (2,-1) {};
	%\node[draw,circle, inner sep=2pt,fill,red] at (4,-1) {};
	%\node[draw,circle, inner sep=2pt,fill,red] at (6,-1) {};
	%\draw[red,ultra thick] (0,0) --  (0,-1);
	%\draw[red,ultra thick] (2,0) --  (2,-1);
	%\draw[red,ultra thick] (4,0) --  (4,-1);
	%\draw[red,ultra thick] (6,0) --  (6,-1);
	
	\node[draw,circle, inner sep=4pt,fill] at (2,0) {};
	%\node[draw,circle, inner sep=2pt,fill,red] at (2,-1) {};
	%\draw[red,ultra thick] (2,0) --  (2,-1);
	\draw[ultra thick] (2,0) -- (3,-2);
	\draw[ultra thick] (4,0) -- (3,-2);
	\draw[ultra thick] (6,0) -- (3,-2);

	\node[draw,circle, inner sep=4pt,fill] at (3,-2) {};
	%\node[draw,circle, inner sep=2pt,fill,red] at (2.5,-3) {};
	%\node[draw,circle, inner sep=2pt,fill,red] at (3.5,-3) {};
	%\draw[red,ultra thick] (3,-2) --  (2.5,-3);
	%\draw[red,ultra thick] (3,-2) --  (3.5,-3);
	\draw[ultra thick] (0,0) -- (3,-2);
\end{tikzpicture}
}
\caption{$I$-poset diagram for $\SL(2,\Q_p)$, $p \equiv 1 \mod 4$}

\end{center}
\end{figure}
\end{eg}

This diagram can be expanded to the order on the orbits given by closure in the Zariski topology.  Each node in the $I$-poset diagram is expanded to the number of orbits corresponding to the torus $A_i$ which is given by the order of $ W_{G_k}(A_i) \slash W_{H_k}(A_i) $.

\begin{eg}
Consider the setting of the previous example.  Each $(\theta,\Q_p)$-split torus corresponds to one orbit, while the $\theta$-fixed torus corresponds to two.  Thus the orbit diagram is given in Figure \ref{sl2qp orbit}.

\begin{figure}[h]

\begin{center}
\resizebox{1in}{.5in}{
\begin{tikzpicture}
	
	\node[draw,circle, inner sep=4pt,fill] at (0,0) {};
	\node[draw,circle, inner sep=4pt,fill] at (2,0) {};
	\node[draw,circle, inner sep=4pt,fill] at (4,0) {};
	\node[draw,circle, inner sep=4pt,fill] at (6,0) {};
	%\node[draw,circle, inner sep=2pt,fill,red] at (0,-1) {};
	%\node[draw,circle, inner sep=2pt,fill,red] at (2,-1) {};
	%\node[draw,circle, inner sep=2pt,fill,red] at (4,-1) {};
	%\node[draw,circle, inner sep=2pt,fill,red] at (6,-1) {};
	%\draw[red,ultra thick] (0,0) --  (0,-1);
	%\draw[red,ultra thick] (2,0) --  (2,-1);
	%\draw[red,ultra thick] (4,0) --  (4,-1);
	%\draw[red,ultra thick] (6,0) --  (6,-1);
	
	\node[draw,circle, inner sep=4pt,fill] at (2,0) {};
	%\node[draw,circle, inner sep=2pt,fill,red] at (2,-1) {};
	%\draw[red,ultra thick] (2,0) --  (2,-1);
	\draw[ultra thick] (2,0) -- (2,-2);
	\draw[ultra thick] (4,0) -- (2,-2);
	\draw[ultra thick] (6,0) -- (2,-2);
	\draw[ultra thick] (2,0) -- (4,-2);
	\draw[ultra thick] (4,0) -- (4,-2);
	\draw[ultra thick] (6,0) -- (4,-2);

	\node[draw,circle, inner sep=4pt,fill] at (2,-2) {};
	\node[draw,circle, inner sep=4pt,fill] at (4,-2) {};

	%\node[draw,circle, inner sep=2pt,fill,red] at (2.5,-3) {};
	%\node[draw,circle, inner sep=2pt,fill,red] at (3.5,-3) {};
	%\draw[red,ultra thick] (3,-2) --  (2.5,-3);
	%\draw[red,ultra thick] (3,-2) --  (3.5,-3);
	\draw[ultra thick] (0,0) -- (2,-2);
	\draw[ultra thick] (0,0) -- (4,-2);

\end{tikzpicture}
}
\caption{Orbit diagram for $\SL(2,\Q_p)$, $p \equiv 1 \mod 4$}
\label{sl2qp orbit}
\end{center}
\end{figure}
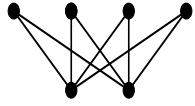
\end{eg}

For a maximal torus $T$, the roots $\Phi(T)$ can be classified in analogy with the case $k=\R$.  $\theta$ acts on the Weyl group, and therefore on the reflections corresponding to the roots.  Thus $\theta$ acts on $\Phi$.  If $\theta(\alpha) = -\alpha$ then $\alpha$ is called a real root, if $\theta(\alpha) = \alpha$ then $\alpha$ is called imaginary, and $\alpha$ is called complex if $\theta(\alpha) \neq \pm \alpha$. 
\begin{defn} Two roots $\alpha, \beta \in \Phi(T)$ are called strongly orthogonal if $(\alpha, \beta) = 0 $ and $\alpha \pm \beta \notin \Phi(T)$.
\end{defn}
\section{Standard $k$-split Tori}
We review some of the results of \cite{tori2}, for our purposes they will be used to determine the domain of the restriction of the generalized complexification map to the $I$-poset.
Let $\mathscr A$ denote the set of $\theta$-stable maximal $k$-split tori.  
\begin{defn}
Let $A_1,A_2 \in \mathscr A$.  Then $(A_1,A_2)$ is called a standard pair if $A_1^- \subset A_2^-$ and $A_2^+ \subset A_1^+$.  $A_1$ is said to be standard with respect to $A_2$.
\end{defn}

The $\theta$-stable maximal $k$-split tori can be arranged into chains of standard pairs as in the algebraically closed case (see \cite{tori1}).
 \begin{thm} Let $A\in \mathscr{A}$ such that $A^{-}$ is maximal $(\theta,k)$-split.  Then there exists a $S \in \mathscr{A}$ standard with respect to $A$ such that $S^+$ is a maximal $k$-split torus of $H$.
\end{thm}

\subsection[$\theta,k)$-singularity]{(\boldmath$\theta,k)$-singularity}
The one dimensional subgroups containing both a $\theta$-fixed $k$-split torus and a $(\theta,k)$-split torus depend heavily on the $k$-structure of the group but can still be parameterized by tori.  An involution defined over $k$ of a connected reductive group $M$ is called split if there exists a maximal $k$-torus which is $(\theta,k)$-split.

\begin{defn} Let $A \in \mathscr A$ and for each $\alpha \in \Phi(A)$ let $\ker(\alpha) = \{ a \in A \st s_{\alpha}(a) =a \}$, where $s_{\alpha}$ is the reflection corresponding to $\alpha$.  Set $G_{\alpha} = Z_G(\ker(\alpha))$.  Then $\alpha$ is called $(\theta,k)$-singular if 
\begin{enumerate}[(a)]
\item $\theta |_{[G_{\alpha},G_{\alpha}]}$ is split
\item $k\krank([G_{\alpha},G_{\alpha}]) =k\krank([G_{\alpha},G_{\alpha}]\cap H)$
\end{enumerate}
\end{defn}

\begin{thm}
Let $A$ be a $\theta$-stable maximal $k$-split torus of $G$ and $\Psi = \{ \alpha_1, \dots, \alpha_r \} \subset \Phi(A)$ a set of strongly orthogonal roots.  Let $G_{\Psi} = G_{{\alpha_1}} \cdots G_{\alpha_r}$.  Then $$[G_{\Psi},G_{\Psi}] = \ds \prod_{i=1}^r [G_{{\alpha_i}},G_{{\alpha_i}}]$$
Moreover, if $\alpha_1, \dots, \alpha_r $ are $(\theta,k)$-singular, then $\theta |_{[G_{\Psi},G_{\Psi}]}$ is $k$-split and $k\krank([G_{\Psi},G_{\Psi}]) = k\krank([G_{\Psi},G_{\Psi}]\cap H)$.
\end{thm}

\section{Generalized Complexification}
\label{GC}

Consider the orbits $\bgh$ over an algebraically closed field.  These are the orbits of a minimal parabolic subgroup acting on a symmetric variety, which can be related to the orbits $\pghk$ of a minimal parabolic $k$-subgroup acting on a symmetric $k$-variety.  A description of how the algebraically closed orbits break up over a subfield is a fundamental question related to the representation theory of the symmetric $k$-varieties.  In this section we approach this problem from the reverse direction, namely by embedding the orbits over a subfield $k$ into the orbits over its algebraic closure $\bar{k}$.  When $k = \R$ this process is the complexification of the real orbits, thus we call the map yielded by the embedding generalized complexification.

Let $P\subset G$ be a minimal parabolic $k$-subgroup and $A \subset P$ a maximal $k$-split torus.  We define the generalized complexification map:
\begin{align*}
\varphi: P_k \backslash G_k / H_k &\rightarrow P\backslash G / H \\
 P_kgH_k &\mapsto P gH
\end{align*}
Recall from Theorem \ref{pghk} that there are several equivalent characterization of the double cosets $\pghk$.  The generalized complexification map $\varphi$ induces maps across all of these equivalent formulations.  Given $v \in V_k$, let $x(v)$ be representative in $N_{G_k}(A)$ such that $v = Z_{G_k}(A) x(v) H_k$.  Then we have an induced map:

\begin{align*}
\varphi_V: V_k &\rightarrow V \\
Z_{G_k}(A) x(v) H_k &\mapsto Z_G(A) x(v) H
\end{align*}

Let $G$ be a $k$-split group.  Then minimal parabolic $k$-subgroups are Borel $k$-subgroups.  In this case we have a simpler generalized complexification map:
\begin{align*}
\varphi: B_k \backslash G_k / H_k &\rightarrow B\backslash G / H \\
 B_kgH_k &\mapsto B gH
\end{align*}
The corresponding induced maps are also simpler in this case since the maximal $k$-split tori are in fact maximal tori.  The generalized complexification of orbits corresponding to the $B_k \times H_k$ action on $G_k$ becomes:
\begin{align*}
\varphi: V_k &\rightarrow V \\
A_k x(v) H_k &\mapsto A x(v) H
\end{align*}

The greatestest simplification, however, occurs in the induced map among the union of quotients of Weyl groups.  Let $\{A_i\}_{i \in I}$ be a set of representatives of the $H_k$-conjugacy classes of $\theta$-stable maximal $k$-split tori.  Then $\{A_i\}_{i \in I}$ corresponds to $\{B_i\}_{i \in I^{\prime}}$, a set of representatives for the $H$-conjugacy classes of $\theta$-stable maximal tori.  This is done in the following manner.  Among the $\{A_i\}$ that correspond to the same $H$-conjugacy class of $\theta$-stable maximal tori, a representative is chosen.  This set is then extended with arbitary representatives of the $H$-conjugacy classes of $\theta$-stable maximal tori not obtained from the $\{A_i\}$.  Therefore the generalized complexification map acts as the identity:
\begin{align*}
\varphi: \ds \bigcup_{i \in I} W_{G_k}(A_i) \slash W_{H_k}(A_i) &\rightarrow \bigcup_{i \in I^{\prime}} W_G(A_i) \slash W_H(A_i) \\
gW_{H_k}(A_i) &\mapsto gW_H(A_i)
\end{align*}

\begin{rem} For groups that are not $k$-split, $\varphi$ still induces a map on the union of Weyl group quotients.  This map is more complicated and involves the introduction of another quotient.
\end{rem}

\subsection{Some Examples}
\label{examples}
In general, the surjectivity of the generalized complexification map depends on both the choice of involution $\theta$ and the field of definition $k$.  The first example of this section illustrates the dependence of the surjectivity of the generalized complexification map on the choice of involution.
\begin{eg} Let $G= \SL(2,\C)$ with real form $G_{\R} = \SL(2,\R)$.  Note that in this case $G$ is $\R$-split.  Let $T$ denote the set of diagonal matrices and $P=B$ the set of upper triangular matrices.
\begin{enumerate}[(a)]

\item Let $\sigma = \Int(\begin{smallmatrix} 0&1 \\ 1 & 0 \end{smallmatrix}) $.  Then from Example \ref{sl2 compute} we have that $\left| B_{\R} \backslash G_{\R} \slash H_{\R} \right|= 4$.  Two orbits correspond to each representative of the $H_{\R}$-conjugacy class of $\sigma$-stable maximal $\R$-split tori, denote the orbits corresponding to $T$ by $\mathscr O_1$ and $\mathscr O_2$.  There is only one orbit corresponding to $T$ over $\C$, therefore the complexification maps $\mathscr O_1$ and $\mathscr O_2$ to the same orbit over $\C$.  Consider $\varphi: B_{\R} \backslash G_{\R} \slash H_{\R} \rightarrow \bgh$.  We can represent the action of $\varphi$ diagrammatically, as in Figure \ref{sl2r inn1 gc}.

\begin{figure}[h]

\begin{center}
$\begin{matrix}
\begin{tikzpicture}
	\node[draw,circle, inner sep=2pt,fill] at (0,0) {};
	\node[draw,circle, inner sep=2pt,fill] at (.5,0) {};
	\node[draw,circle, inner sep=2pt,fill] at (.5,-.5) {};
	\node[draw,circle, inner sep=2pt,fill] at (0,-.5) {};

	\draw[ thick] (0,0) -- (0,-.5);
	\draw[ thick] (0,0) -- (.5,-.5);
	\draw[ thick] (.5,0) -- (0,-.5);
	\draw[ thick] (.5,0) -- (.5,-.5);

\end{tikzpicture} 
& \overset{\varphi}{\xrightarrow{\hspace{1 in}} }

&\begin{tikzpicture}
	\node[draw,circle, inner sep=2pt,fill] at (.25,0) {};
	\node[draw,circle, inner sep=2pt,fill] at (.5,-.5) {};
	\node[draw,circle, inner sep=2pt,fill] at (0,-.5) {};

	\draw[ thick] (.25,0) -- (0,-.5);
	\draw[ thick] (.25,0) -- (.5,-.5);

\end{tikzpicture}

\\
\SL(2,\mathbb R) & &\SL(2,\mathbb C) &  \\
\theta = \intone&&
\end{matrix}
$ 
\caption{Generalized complexification of $\SL(2,\R)$, $\theta=\Int\usebox{\smlmat}$}
\label{sl2r inn1 gc}
\end{center}

\end{figure}

\item Let $\theta = \Int(\begin{smallmatrix} 0&1 \\ -1 & 0 \end{smallmatrix}) $.  Then from Example \ref{sl2 compute} we have that $\left| B_{\R} \backslash G_{\R} \slash H_{\R} \right|= 1$. The complexification map has a cokernel, indicated by the empty nodes in Figure \ref{sl2r inn-1 gc}

\begin{figure}[h]

\begin{center}
$\begin{matrix}
\begin{tikzpicture}
	\node[draw,circle, inner sep=2pt,fill] at (0,0) {};
	\node[draw,circle, inner sep=2pt,fill,color=white] at (0,-.5) {};

\end{tikzpicture}
& \overset{\varphi}{\xrightarrow{\hspace{1 in}} }

&\begin{tikzpicture}
	\node[draw,circle, inner sep=2pt,fill] at (.25,0) {};
	\node[draw,circle, inner sep=2pt] at (.5,-.5) {};
	\node[draw,circle, inner sep=2pt] at (0,-.5) {};

	\draw[ thick, dashed] (.25,0) -- (0,-.5);
	\draw[ thick, dashed] (.25,0) -- (.5,-.5);

\end{tikzpicture}

\\
\SL(2,\mathbb R) & &\SL(2,\mathbb C) &  \\
\theta = \Int\left(\begin{smallmatrix} 0 & 1 \\ -1 & 0 \end{smallmatrix}\right)&&
\end{matrix}
$ 
\caption{Generalized complexification of $\SL(2,\R)$, $\theta=\Int\usebox{\smlmatneg}$}
\label{sl2r inn-1 gc}
\end{center}

\end{figure}

\end{enumerate}
\end{eg}

Recall that there infinite number of orbits $B_{\Q} \backslash G_{\Q} \slash H_{\Q}$ for $G_{\Q} = \SL(2,\Q)$.  The complexification of these orbits is quite similar to the complexification of the real orbits, as shown in Figure \ref{sl2q inn1 gc} and Figure \ref{sl2q inn-1 gc}.

\begin{figure}

\begin{center}
$\begin{matrix}
\begin{tikzpicture}
	\node[draw,circle, inner sep=2pt,fill] at (0,0) {};
	\node[draw,circle, inner sep=2pt,fill] at (.5,0) {};
	\node[draw,circle, inner sep=2pt,fill] at (1,0) {};
	\node[draw,circle, inner sep=2pt,fill] at (1.5,0) {};
	\node[draw,circle, inner sep=2pt,fill] at (0.5,-.5) {};
	\node[draw,circle, inner sep=2pt,fill] at (1,-.5) {};

	\node[draw=none] (ellipsis1) at (2,-.25) {$\cdots$};
	\node[draw=none] (ellipsis2) at (-.5,-.25) {$\cdots$};

	\draw[thick] (0,0) -- (0.5,-.5);
	\draw[thick] (.5,0) -- (0.5,-.5);
	\draw[thick] (1,0) -- (0.5,-.5);
	\draw[thick] (1.5,0) -- (0.5,-.5);
	\draw[thick] (0,0) -- (1,-.5);
	\draw[thick] (.5,0) -- (1,-.5);
	\draw[thick] (1,0) -- (1,-.5);
	\draw[thick] (1.5,0) -- (1,-.5);

\end{tikzpicture}
& \overset{\varphi}{\xrightarrow{\hspace{1 in}} }

&\begin{tikzpicture}
	\node[draw,circle, inner sep=2pt,fill] at (.25,0) {};
	\node[draw,circle, inner sep=2pt,fill] at (.5,-.5) {};
	\node[draw,circle, inner sep=2pt,fill] at (0,-.5) {};

	\draw[ thick] (.25,0) -- (0,-.5);
	\draw[ thick] (.25,0) -- (.5,-.5);

\end{tikzpicture}

\\
\SL(2,\mathbb Q) & &\SL(2,\mathbb C) &  \\
\theta = \intone&&
\end{matrix}
$ 
\caption{Generalized complexification of $\SL(2,\Q)$, $\theta=\Int\usebox{\smlmat}$}
\label{sl2q inn1 gc}
\end{center}

\end{figure}

\begin{figure}

\begin{center}
$\begin{matrix}
\begin{tikzpicture}
	\node[draw,circle, inner sep=2pt,fill] at (0,0) {};
	\node[draw,circle, inner sep=2pt,fill] at (.5,0) {};
	\node[draw,circle, inner sep=2pt,fill] at (1,0) {};
	\node[draw,circle, inner sep=2pt,fill] at (1.5,0) {};
	\node[draw=none] (ellipsis1) at (2,-.25) {$\cdots$};
	\node[draw=none] (ellipsis1) at (-.5,-.25) {$\cdots$};

	\node[draw,circle, inner sep=2pt,fill,color=white] at (0,-.5) {};

\end{tikzpicture}
& \overset{\varphi}{\xrightarrow{\hspace{1 in}} }

&\begin{tikzpicture}
	\node[draw,circle, inner sep=2pt,fill] at (.25,0) {};
	\node[draw,circle, inner sep=2pt] at (.5,-.5) {};
	\node[draw,circle, inner sep=2pt] at (0,-.5) {};

	\draw[ thick, dashed] (.25,0) -- (0,-.5);
	\draw[ thick, dashed] (.25,0) -- (.5,-.5);

\end{tikzpicture}

\\
\SL(2,\Q) & &\SL(2,\mathbb C) &  \\
\theta = \Int\left(\begin{smallmatrix} 0 & 1 \\ -1 & 0 \end{smallmatrix}\right)&&
\end{matrix}
$ 
\caption{Generalized complexification of $\SL(2,\Q)$, $\theta=\Int\usebox{\smlmatneg}$}
\label{sl2q inn-1 gc}
\end{center}

\end{figure}

To illustrate the dependence on the base field $k$, consider $k = \mathbb Q(i)$. In this case we obtain surjectivty for $\theta = \Int\left(\begin{smallmatrix} 0 & 1 \\ -1 & 0 \end{smallmatrix}\right)$ since in this case $H=G^{\theta}$ is a torus which splits over $k$.

Given a set $S \subset \Aut(G)$, we say two involutions $\theta$, $\sigma \in \Aut(G)$ are $S$-isomorphic if there exists $\gamma \in S$ such that $\theta = \gamma \sigma \gamma^{-1}$.  In this case we say that $\theta$ and $\sigma$ are $S$-isomorphic by $\gamma$. The double coset decomposition is dependent only on the isomorphy class of involution.

\begin{prop}
\label{fixed pts}
Let $\theta$, $\sigma$ be involutions that are $\Aut(G)$-isomorphic by $\gamma$, and denote their fixed point groups $H_1 = G^{\theta}$ and $H_2 = G^{\sigma}$.  Then $H_2 = \gamma^{-1}(H_1)$.
\end{prop}
\begin{proof} Let $h \in H_1$.  Then $\sigma(\gamma^{-1}(h)) = \gamma^{-1} \theta \gamma(\gamma^{-1}(h)) = \gamma^{-1} \theta(h) = \gamma^{-1}(h)$.  Therefore $\gamma^{-1}(H_1) \subset H_2$.  Since $\gamma$ is one-to-one, we have that $H_2 = \gamma^{-1}(H_1)$.
\end{proof}

\begin{cor}
\label{fixed pts cor}
Assume the hypotheses of Proposition \ref{fixed pts}. If $\theta, \ \gamma$ are $\Int(G)$-isomorphic, then $H_1, \ H_2$ are conjugate.
\end{cor}
\begin{proof}Apply Proposition \ref{fixed pts} to the case $\gamma = \Int(g)$ for some $g \in G$.
\end{proof}

\begin{thm}Let $\theta, \ \sigma \in \Aut(G)$ be $\Int(G)$-isomorphic involutions.  If $H_1 = G^{\theta}$ and $H_2 = G^{\sigma}$, then there exist Borel subgroups $B_1$ and $B_2$ such that $B_1 \backslash G \slash H_1 \cong B_2 \backslash G \slash H_2$.  Furthermore, $B_1$ and $B_2$ are $G$-conjugate.
\end{thm}
\begin{proof}Suppose $\theta = \Int(g) \sigma \Int(g)^{-1}$.  From Corollary \ref{fixed pts cor} we have that $H_2 = g^{-1}H_1g$.  Let $B_2 = g^{-1} B_1 g$.  Given a double coset $B_1xH_1 \in B_1 \backslash G \slash H_1$, we compute $\Int(g^{-1})(B_1xH_1) = \Int(g^{-1})(B_1)\Int(g^{-1})(x)\Int(g^{-1})(H_1) = B_2 g^{-1}xg H_2 \in B_2 \backslash G \slash H_2$.  The inverse map is given by $\Int(g)$, so we have $B_1 \backslash G \slash H_1 \cong B_2 \backslash G \slash H_2$.
\end{proof}

The isomorphy of the double coset decompositions extends to $k$-isomorphy, using the same proofs.
\begin{thm}Let $\theta, \ \sigma \in \Aut(G,G_k)$ be $\Int(G,G_k)$-isomorphic involutions with fixed point groups $H_1$ and $H_2$ respectively.  Then there exist Borel subgroups $B_1$ and $B_2$ such that $(B_1)_k \backslash G_k \slash (H_1)_k \cong (B_2)_k \backslash G_k \slash (H_2)_k$.
\end{thm}

In light of Theorem \ref{pghk}, the double cosets $\bghk$ are parameterized by the $H_k$-conjugacy of $\theta$-stable maximal $k$-split tori.  The $H_k$-conjugacy classes have not been fully classified except in a number of specific cases, notably $\SL(2,k)$.  However, for algebraically closed fields a complete classification was achieved in \cite{tori1}.  Conveniently, the characterization of surjectivity of the generalized complexification map depends only on the $H$-isomorphy classes of $\theta$-stable maximal tori.

If we fix $\theta$-stable maximal $k$-split torus $A$, we can restrict the generalized complexification map \[ \varphi: \ds \bigcup_{i \in I} W_{G_k}(A_i) \slash W_{H_k}(A_i) \rightarrow \bigcup_{i \in I^{\prime}} W_G(A_i) \slash W_H(A_i) \] to the Weyl group quotient corresponding to $A$:
\begin{equation}
\label{phiA}
 \varphi_A: \ds  W_{G_k}(A) \slash W_{H_k}(A) \rightarrow  W_G(A) \slash W_H(A) 
\end{equation}
Then we can consider the surjectivity of $\varphi_A$.  The following lemma shows that this map is in fact surjective in all cases.

\begin{lemma} The map $\varphi_A$ of Equation \ref{phiA} is surjective.
\end{lemma}
\begin{proof} It is clear that $W_{H_k}(A) \subset W_H(A)$.  Given $gW_H(A) \in W_G(A) \slash W_H(A)$, the fiber is nonempty since $g$ has a representative in $W_{G_k}(A)$.
\end{proof}

Thus surjectivity of the map between indexing sets of $\theta$-stable maximal tori is sufficient to ensure surjectivity of the generalized complexification map.

Here we recall the general Cayley transform, whose construction is quite similar to the real case.  

Fix a maximal $(\theta,k)$-split torus $A$ of $G$.  Then $A$ has a root system $\Phi(A)$ and for each $\alpha \in \Phi(A)$ we can define the root group $G_{\alpha} = Z_G(\ker(\alpha))$.  Then the commutator $[G_{\alpha},G_{\alpha}]$ is a semisimple group of rank 1, isomorphic to $\SL(2,k)$.  Therefore $\theta|_{[G_{\alpha},G_{\alpha}]}$ is inner.  If $[G_{\alpha},G_{\alpha}]$ contains a nontrivial $\theta$-fixed torus $S$, we define a map $\eta = \Int\left( \begin{smallmatrix} 1 & \frac{-1}{2} \\ 1 & \frac{1}{2} \end{smallmatrix} \right)$ that acts on $S$.  As in the real case, $\eta$ maps $S$ to a $\theta$-split torus of $[G_{\alpha},G_{\alpha}]$. 

\begin{lemma}
\label{theta restr}
$S$ is $k$-split if and only if $\theta|_{[G_{\alpha},G_{\alpha}]} \cong \intone$
\end{lemma}
\begin{proof}Since $[G_{\alpha},G_{\alpha}] \cong \SL(2,k)$, this follows directly from \cite{beunpaper}.
\end{proof}
Furthermore, $S$ is $k$-split implies $\eta(S)$ is also $k$-split.  We extend the action of $\eta$ to the rest of the torus trivially.

In order to discuss the surjectivity of the generalized complexification map we should have a description of the $I$-posets in the image of the generalized complexfication map, namely the $I$-posets for algebraically closed fields.  This was carried out in \cite{tori1}.  In particular, we know when $H$ contains a maximal torus.

\begin{thm} 
\label{alg closed ranks}
$\theta \in \Int(G)$ if and only if $\rank(G) = \rank(H)$.
\end{thm}

\begin{lemma}
\label{lemma:exist tk-split root}
Suppose $H$ contains a nontrivial $k$-split torus $S$.  Then $\Phi(S)$ consists of $(\theta,k)$-singular roots.
\end{lemma}
\begin{proof} Consider $\alpha \in \Phi(S)$, then $\theta(\alpha) = \alpha$.  Construct the corresponding root group $G_{\alpha} = Z_G(\ker(\alpha))^{\textdegree}$.  Then by Lemma \ref{theta restr} we have that $\theta|_{[G_{\alpha},G_{\alpha}]} \cong \intone$.  Therefore $[G_{\alpha},G_{\alpha}]$ contains a $\theta$-fixed $k$-split torus then can be flipped to a $(\theta,k)$-split torus in $[G_{\alpha},G_{\alpha}]$ via the Cayley transform $\eta$.  Therefore $\alpha$ is $(\theta,k)$-singular.
\end{proof}

\begin{cor}Let $S\subset H$ be a $k$-split torus.  There exists a maximal orthogonal subset of roots of torus lying in $H$ such that each root is $(\theta,k)$-singular.
\end{cor}

We can iterate this process, performing successive Cayley transforms in the root groups of a set of strongly orthogonal roots.  
\begin{lemma} 
\label{lemma:flipping}
 Assume $\rank(G)=n$ and let $S \subset H$ be a $k$-split torus of $H$ and suppose $\Psi(S) = \{ \alpha_1, \dots, \alpha_r \} \subset \Phi(S)$ is a maximal set of strongly orthogonal roots.  Then $G$ contains a $(\theta,k)$-split torus of dimension $n-r$.
\end{lemma}
\begin{proof}  We use induction on $r$.  We may assume the $\Psi(S)$ consists of $(\theta,k)$-singular roots.  The case $r=1$ is carried out explicitly in the proof of Lemma \ref{lemma:exist tk-split root}.  Now assume $r > 1$ and let ${^{\alpha_1}S} \subset H$ be the subtorus lying in $[G_{\alpha_1},G_{\alpha_1}]$.  Then $S = ({^{\alpha_1}S})\tilde{S}$, where $\tilde{S} \subset S$ denotes the factor of $S$ such that $[G_{\alpha_1},G_{\alpha_1}]\cap \tilde{S} = \pm\id$.  Then $\tilde{S}$ is a $k$-split torus in $H$ so Lemma \ref{lemma:exist tk-split root} applies and $|\Psi(S)|=r-1$.
\end{proof}

\begin{lemma} Let $S$ be a maximal $k$-split torus of $H$.  Then $Z_{G}(S)$ contains a maximal $k$-split torus of $G$.
\end{lemma}
\begin{proof} 

Let $S \subset H$ be a maximal $k$-split torus of $H$.  Consider $$A_1 = \displaystyle \left (\bigcap_{\substack{\alpha \in \Phi(T) \\ \alpha \perp \Phi(S)}} \ker(\alpha) \right)^{\circ}$$
Since the $-1$-eigenspace of $\theta$ is orthogonal to $H$, $\pi(A_1)$ (where $\pi$ is the usual projection) contains a maximal torus of $G/H$, let $A_1^-$ be the inverse image of this torus.  Then there exists a subtorus of $A_1^-$ that is maximal $k$-split in $G/H$.  Therefore $S\cdot A_1^-$ is a maximal $k$-split torus of $G$.
\end{proof}

\begin{lemma}
\label{preserve dim}
 Given a torus $S$, $\dim \varphi(S)^+ = \dim S^+$ and $\dim \varphi(S)^- = \dim S^-$
\end{lemma}

\begin{thm}  
\label{main thm}
Let $G$ be a $k$-split group.  Then the generalized complexification map $\varphi$ is surjective if and only if $k\krank(H) = \rank(H)$.
\end{thm}
\begin{proof} First suppose $k\krank(H) \neq \rank(H)$.  Let $\{A_i\}_{i \in I}$ be a set of representatives for the $H_k$-conjugacy classes of $\theta$-stable maximal $k$-split tori and let $\{B_l\}_{l \in L}$ be a set of representatives for the $H$-conjugacy classes of $\theta$-stable maximal tori.  Choose $ A \in \{A_i\}_{i \in I}$ such that $A^+$ is maximal and $B \in \{B_l\}_{l \in L}$ such that $B^+$ is maximal.  Then by assumption $\dim B^+ > \dim A^+$.  In light of \ref{preserve dim}, $B$ can have no preimage under the induced complexification map on the torus poset.  Thus none of the orbits counted by $W_G(B) \slash W_H(B)$ are in the image of $\varphi$, so $\varphi$ is not surjective.

Next assume that $k\krank(H) = \rank(H)$.  Then let $S$ be a maximal $k$-split torus of $H$, and let $T\subset Z_G(S)$ be a maximal $k$-split torus of $G$.  Then the $\Phi(T)$ and and the restricted root system $\Phi_0(T)$ consist of the same roots, thus they have the same maximal orthogonal set of roots and all such roots are $(\theta,k)$-singular.  Thus we have surjectivity in the $I$-poset and thus $\varphi$ is surjective.
\end{proof}

\begin{cor}Let $G$ be a $k$-split group, suppose the generalized complexification map $\varphi: \bghk \rightarrow \bgh$ is surjective, and let $A \subset G$ be a $k$-split torus.  Define $G_1 = Z_G(A)$, $B_1 \subset G_1$ a Borel sugroup, and $H_1 = H \cap G_1$.  Then the restriction $\varphi|_{G_1}$ is surjective.
\end{cor}
\begin{proof} $G_1$ is connected and reductive since $A$ is a $k$-split torus.  Moreover, $B_1$ is contained in a Borel subgroup of $G$, i.e. $B_1 = B \cap G_1$ for $B$ a Borel subroup, and therefore $(B_1)_k = (B \cap G_1)_k$.   Thus the orbits $(B_1) g (H_1) \in (B_1) \backslash (G_1) \slash (H_1)$ can be embedded in the orbits $\bghk$.  Therefore there is a preimage of $B_1 g H_1$ in $\bghk$.  Then $\varphi^{-1}(B_1 g H_1) \cap (G_1)_k$ is nonempty, so surjetivity is achieved.
\end{proof}

The surjectivity of the generalized complexification  map implies that there exists a decomposition of $G_k$ that is as far from the Iwasawa decomposition as possible.   We conclude with some examples.

\begin{eg}Let $G = \SL(n,\C)$, $G_{\R}= \SL(n,\R)$, $\theta(g) = (g^T)^{-1}$ for all $g \in G$.  Then $H_{\R} = \SO(2,\R)$, which is compact.  Therefore $\R\krank(H) = 0$, so $\varphi$ is not surjective.  In fact, from the Iwasawa decomposition we can deduce the generalized complexification diagram on the $I$-poset found in Figure \ref{slnr cartan}.
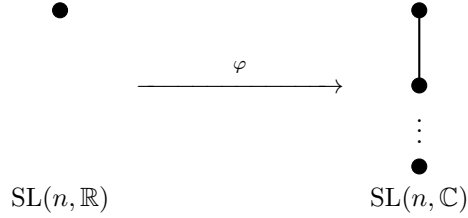
\begin{figure}[h]

\begin{center}
$\begin{matrix}
\begin{tikzpicture}
	\node[draw,circle, inner sep=2pt,fill] at (0,0) {};
	\node[draw,circle, inner sep=2pt,fill,color=white] at (0,-1) {};
\end{tikzpicture}
& \overset{\varphi}{\xrightarrow{\hspace{1 in}} }
&\begin{tikzpicture}
	\node[draw,circle, inner sep=2pt,fill] at (.25,0) {};
	\node[draw,circle, inner sep=2pt,fill] at (.25,-1) {};
	\draw[ thick, ] (.25,0) -- (.25,-1);
\end{tikzpicture}
\\
&&\vdots \\
&&\begin{tikzpicture}
	\node[draw,circle, inner sep=2pt,fill] at (.25,) {};
\end{tikzpicture}\\
\SL(n,\mathbb R) & &\SL(n,\mathbb C) &  \\
\end{matrix}
$ 
\caption{Generalized complexification of $\SL(n,\R)$, $\theta(g) = (g^T)^{-1}$}
\label{slnr cartan}
\end{center}
\end{figure}
\end{eg}

\begin{eg} Let $G_k = \SL(n,k)$ and $\theta = \Int(I_{n-i,i})$, where 
\[
I_{n-i,i}=
\begin{bmatrix}
J & 0 \\
0 & I_{n-2i}
\end{bmatrix},
\]
with $I_{n-2i}$ denoting the $n-2i \times n-2i$ identity matrix and 
\[
J= \begin{bmatrix}
0 & \cdots &  0 & 0 & 1 \\
0 & \cdots &  0 & 1 & 0 \\
\vdots & \iddots & \iddots & \iddots & \vdots \\
0& 1 & 0& \cdots & 0 \\
1&  0 & 0& \cdots&0
\end{bmatrix}.
\]
The fixed point group $H$ of $\theta$ consists of matrices of the form
$
\begin{bmatrix}
A&B\\C&D
\end{bmatrix}
$
where
$A$ has dimensions $2i \times 2i$ and satisfies $JAJ =A$, $B$ has dimenstions $2i \times n-2i$ and satisfies $JB = B$, $C$ has dimensions $n-2i \times 2i$ and satisfies $CJ=C$ and $D \in \GL(n-2i)$.  Then one checks that 
\[ P=
\begin{bmatrix}
a_1 & 0 & \cdots & \cdots &0 & b_1 \\
0& \ddots &&&\iddots & 0 \\
\vdots  &&a_i&b_i &&\vdots \\
\vdots  &&b_i&a_i &&\vdots \\
0& \iddots &&&\ddots & 0 \\
b_1 & 0 & \cdots & \cdots &0 & a_1 \\
\end{bmatrix}
\]
satisfies $JPJ = P$, and has eigenvalues $a_j \pm b_j$ for $j\in \{1,2,\dots,i\}$.  Then elements of the form
\[
\begin{bmatrix}
P & 0 & \cdots & 0 \\
0 & x_1 &&\vdots \\
\vdots & & \ddots & 0 \\
0& \cdots &0 & x_{n-2i}
\end{bmatrix},
\]
with $P$ as above and the $x_j$ elements of the field, form a torus.  This torus is $k$-split, and is maximal because its dimension is $n-1$.  Therefore $k\krank(H) = \rank(H)$, and thus the generalized complexification map is surjective for all choices of $n$, $i$, and $k$.

\end{eg}

\begin{eg}
Let $G_k = \SL(n,k)$ such that $n=2m$ is even.  Furthermore, let $\theta = \Int(L_{2m,x})$, where $ x \not\equiv 1 \mod (k^*)^2$, $L_{2m,x} = \diag (L_1, L_2, \dots , L_m)$, and 
\[
L_i =
\begin{bmatrix}
0&1\\x&0
\end{bmatrix}.
\]

The fixed point group of $\theta$ is given by $[A_{ij}]$, $i,j \in \{1,2, \dots, m\}$, where 
\[
A_{ij} =
\begin{bmatrix}
a_{ij} & b_{ij} \\ xb_{ij} & a_{ij}
\end{bmatrix}.
\]

A maximal $k$-split torus of $H$ is given by:
$$A =\left\{ \left. \diag\{ a_1 , a_1, \dots, a_{m}, a_{m}\} \ \right| \ a_1^2 \cdots  a_{m}^2 =1 \right\}.$$
Therefore we do not have surjectivity in these cases.  Consider the centralizer in $H_k$ of $A$:
$$Z_{H_k}(A) = \begin{pmatrix}
K_1 &0&\cdots & 0 \\
0 & K_2 &&\vdots \\
\vdots && \ddots &0 \\
0 & \cdots &0&K_{\frac{n}{2}}
\end{pmatrix}$$
where each $K_i$ is a $2\times 2$ matrix:
$$ \begin{pmatrix}
a_i & b_i \\
xb_i & a_i
\end{pmatrix}$$

Then $Z_{H_k}(A)$ is diagonalizable over $\tilde{k}=k(\sqrt{x})$ since its eigenvalues are $a_i \pm b_i\sqrt{x}$, $i=1, \ 2,\  \dots, \ \frac{n}{2}$.  Thus we have surjectivity of the generalized complexification over the quadratic extension field $\tilde{k}$.
\end{eg}

\bibliographystyle{alpha}
\bibliography{Biblio}

%% Authors are advised to submit their bibtex database files. They are
%% requested to list a bibtex style file in the manuscript if they do
%% not want to use model1-num-names.bst.

\end{document}